\documentclass[11pt,oneside,english]{amsart}
\usepackage[]{fontenc}
\usepackage[latin9]{inputenc}
\usepackage{babel}
\usepackage{array}
\usepackage{multirow}
\usepackage{amsthm}
\usepackage{amstext}
\usepackage{amssymb}
\usepackage[unicode=true,pdfusetitle,
 bookmarks=true,bookmarksnumbered=false,bookmarksopen=false,
 breaklinks=false,pdfborder={0 0 1},backref=false,colorlinks=false]
 {hyperref}

\makeatletter

\providecommand{\tabularnewline}{\\}

\numberwithin{equation}{section}
\numberwithin{figure}{section}
\theoremstyle{plain}
\newtheorem{thm}{\protect\theoremname}[section]
  \theoremstyle{plain}
  \newtheorem{lem}[thm]{\protect\lemmaname}
  \theoremstyle{plain}
  \newtheorem{prop}[thm]{\protect\propositionname}
  \theoremstyle{plain}
  \newtheorem{cor}[thm]{\protect\corollaryname}
  \theoremstyle{remark}
  \newtheorem*{rem*}{\protect\remarkname}

\usepackage[a4paper]{geometry}

\usepackage{graphicx}

\makeatother

  \providecommand{\corollaryname}{Corollary}
  \providecommand{\lemmaname}{Lemma}
  \providecommand{\propositionname}{Proposition}
  \providecommand{\remarkname}{Remark}
\providecommand{\theoremname}{Theorem}

\begin{document}

\title{Parity preservation of $K$-types under theta correspondence}

\author{Xiang Fan}

\email{fanx8@mail.sysu.edu.cn}

\address{School of Mathematics, Sun Yat-sen University, Guangzhou 510275,
China}


\begin{abstract}
This note shows  a property of degree-parity preservation for $K$-types
under Howe\textquoteright s theta correspondence. As its application,
 we deduce the preservation of parity of all $K$-types occurring
in an arbitrary irreducible $(\mathfrak{g},K)$-module of any Lie
group in reductive dual pairs.
\end{abstract}

\maketitle

\section{Degree-parity Preservation}

\emph{Howe's duality correspondence} of irreducible admissible representations
for reductive dual pairs was introduced by Roger Howe in the 1970s.
It is also called the \emph{theta correspondence} as an extension
of Weil's representation-theoretic approach to classical $\theta$-series.
In this short note, a degree-parity preservation property (Theorem
\ref{thm:deg-parity}) for $K$-types is shown for the local theta
correspondence of reductive Lie groups, with an interesting application,
Theorem \ref{thm:Main} (or Theorem \ref{thm:gen} in its general
form), which asserts the preservation of parity of all $K$-types
occurring in an arbitrary irreducible admissible representation of
a Lie group in reductive dual pairs.

For a continuous admissible representation of a real reductive Lie
group $G$, as we focus on its $K$-spectrum, we may replace it by
its Harish-Chandra $(\mathfrak{g},K)$-module (consisting of its $K$-finite
smooth vectors). Here $K$ is a maximal compact subgroup of $G$,
and $\mathfrak{g}$ is the complexified Lie algebra of $G$. Throughout
this note, we use upper case Latin letters (e.g., $G$, $G'$, $K$,
$T$) to denote Lie groups, and the corresponding lower case Gothic
letters (e.g., $\mathfrak{g}$, $\mathfrak{g}'$, $\mathfrak{k}$,
$\mathfrak{t}$) to indicate their complexified Lie algebras.

By a \emph{$K$-module} we mean a continuous representations of $K$,
and by a \emph{$K$-type} we mean an equivalence class of irreducible
$K$-modules (which are automatically finite-dimensional and unitary).
Let $\mathcal{R}(K)$ denote the set of all $K$-types. By an abuse
of notation, for a $K$-type $\sigma\in\mathcal{R}(K)$, we also understand
$\sigma$ as an irreducible $K$-module (up to equivalence).

For a $(\mathfrak{g},K)$-module $V$, its \emph{$K$-spectrum} is
the $K$-module decomposition
\[
V\simeq\bigoplus_{\sigma\in\mathcal{R}(K)}W_{\sigma}^{\oplus m(\sigma,V)},
\]
where $W_{\sigma}$ is an underlying space of a $K$-type $\sigma$,
and
\[
m(\sigma,V)=\dim\mathrm{Hom}_{K}(W_{\sigma},V)=\dim\mathrm{Hom}_{K}(V,W_{\sigma})
\]
is the \emph{multiplicity} of $\sigma$ in $V$. We say that ``$\sigma$
occurs in $V$'' if $m(\sigma,V)\neq0$.  Denote the set of all
$K$-types occurring in $V$ by $\mathcal{R}(K,V)=\{\sigma\in\mathcal{R}(K):m(\sigma,V)\neq0\}$.
A finitely generated $(\mathfrak{g},K)$-module $V$ is called \emph{admissible}
if every $\sigma\in\mathcal{R}(K)$ has finite multiplicity in $V$.
It is well-known that every irreducible $(\mathfrak{g},K)$-module
is admissible.

A real\emph{ reductive dual pair} is a pair $(G,G')$ of closed reductive
subgroups of $\mathbf{Sp}=Sp_{2N}(\mathbb{R})$ (for some $N$) such
that they are mutual centralizers of each other. For a subgroup $E\subseteq\mathbf{Sp}$,
let $\widetilde{E}$ denote its preimage in the \emph{metaplectic
cover} (the unique non-trivial two-fold central extension) $\widetilde{\mathbf{Sp}}$
of $\mathbf{Sp}$. Indeed we have short exact sequence:
\[
1\to\mu_{2}\to\widetilde{\mathbf{Sp}}\to\mathbf{Sp}\to1,
\]
where $\mu_{2}=\mathrm{Ker}(\widetilde{\mathbf{Sp}}\to\mathbf{Sp})$
is the finite group of order $2$. Take the Segal-Shale-Weil\emph{
oscillator representation} (c.f. \cite{Shale1962linear,Weil1964certains,LionVergne1980theweil})
$\omega$ of $\widetilde{\mathbf{Sp}}$ (associated to the character
of $\mathbb{R}$ that sends $t$ to $\exp(2\pi\sqrt{-1}t)$). Let
$\mathfrak{sp}$ denote the complexified Lie algebra of $\mathbf{Sp}$,
and take $\mathbf{U}=U(N)$ as a maximal compact subgroup of $\mathbf{Sp}$.
Fock model realizes the $(\mathfrak{sp},\widetilde{\mathbf{U}})$-module
of $\omega$ on the space $\mathcal{F}=\mathrm{Poly}(\mathbb{C}^{N})$
of complex polynomials on $\mathbb{C}^{N}$.  

Assume that $G$ and $G'$ are embedded in $\mathbf{Sp}$ in such
a way that  $K=\mathbf{U}\cap G$ and $K'=\mathbf{U}\cap G'$ are
maximal compact subgroups of $G$ and $G'$ respectively. Hence $\widetilde{K}$
and $\widetilde{K'}$ are maximal compact subgroups of $\widetilde{G}$
and $\widetilde{G'}$ respectively. Let $\mathfrak{g}$ and $\mathfrak{g}'$
be the complexified Lie algebras of $G$ and $G'$ respectively.
\begin{lem}
[{Algebraic version of local theta correspondence over $\mathbb{R}$
\cite[Theorem 2.1]{Howe1989transcending}}]\label{lem:HoweDual}For
a real reductive dual pair $(G,G')$, if $\pi$ is an irreducible
$(\mathfrak{g},\widetilde{K})$-module, then
\begin{eqnarray*}
\mathcal{F}\bigg/\bigcap_{T\in\mathrm{Hom}_{\mathfrak{g},\widetilde{K}}(\mathcal{F},\pi)}\mathrm{Ker}(T) & \simeq & \pi\otimes\Theta(\pi),
\end{eqnarray*}
where $\Theta(\pi)$ is a finitely generated admissible $(\mathfrak{g}',\widetilde{K'})$-module.
Moreover, if $\Theta(\pi)$ is non-zero, then it has a unique irreducible
$(\mathfrak{g}',\widetilde{K'})$-quotient $\theta(\pi)$ called the
``theta lift'' of $\pi$.\end{lem}
\begin{itemize}
\item The intersection in the quotient may be the whole $\mathcal{F}$.
This case happens $\Leftrightarrow$ $\mathrm{Hom}_{\mathfrak{g},\widetilde{K}}(\mathcal{F},\pi)=0$
$\Leftrightarrow$ $\Theta(\pi)=0$. In this case, let the theta lift
$\theta(\pi)=0$.
\item  When $\Theta(\pi)\neq0$, the ``uniqueness'' of $\theta(\pi)$
means that $\Theta(\pi)$ has a unique maximal proper sub-$(\mathfrak{g}',\widetilde{K'})$-module,
and it is the kernel of $\Theta(\pi)\to\theta(\pi)$.
\end{itemize}
 Recall that the oscillator representation $\omega$ is the direct
sum of two irreducible unitary representations of $\widetilde{\mathbf{Sp}}$,
and we have a $(\mathfrak{sp},\widetilde{\mathbf{U}})$-module decomposition
$\mathcal{F}=\mathcal{F}_{0}\oplus\mathcal{F}_{1}$, where $\mathcal{F}_{i}$
is the linear span of all homogeneous polynomials of degree $\equiv i$
$(\mathrm{mod}\ 2)$ in $\mathcal{F}=\mathrm{Poly}(\mathbb{C}^{N})$,
for $i\in\{0,1\}$.  Then as $(\mathfrak{g}',\widetilde{K}')$-modules
$\Theta(\pi)=\Theta_{0}(\pi)\oplus\Theta_{1}(\pi)$, with 
\[
\mathcal{F}_{i}\bigg/\bigcap_{T\in\mathrm{Hom}_{\mathfrak{g},\widetilde{K}}(\mathcal{F}_{i},\pi)}\mathrm{Ker}(T)\simeq\pi\otimes\Theta_{i}(\pi).
\]

\begin{prop}
\label{prop:either0} Let $\pi$ be an irreducible $(\mathfrak{g},\widetilde{K})$-module.
Then $\Theta_{i}(\pi)=0$ for at least one $i\in\{0,1\}$.  In other
words, $\mathrm{Hom}_{\mathfrak{g},\widetilde{K}}(\mathcal{F}_{i},\pi)=0$
for this $i$.\end{prop}
\begin{proof}
Otherwise, $\Theta(\pi)=\Theta_{0}(\pi)\oplus\Theta_{1}(\pi)$ with
both $\Theta_{i}(\pi)$ non-zero finitely generated admissible $(\mathfrak{g}',\widetilde{K'})$-modules.
By Zorn's Lemma,  $\Theta_{i}(\pi)$ contains a maximal proper sub-$(\mathfrak{g}',\widetilde{K'})$-module
$\Pi_{i}$. Then both $\Theta(\pi)/(\Pi_{0}\oplus\Theta_{1}(\pi))$
and $\Theta(\pi)/(\Theta_{0}(\pi)\oplus\Pi_{1})$ are non-zero irreducible
 $(\mathfrak{g}',\widetilde{K'})$-quotients,   in contradiction
with the uniqueness in Lemma \ref{lem:HoweDual}.\end{proof}
\begin{cor}
\label{cor:disjoint}$\mathcal{R}(\widetilde{K},\mathcal{F}_{0})\cap\mathcal{R}(\widetilde{K},\mathcal{F}_{1})=\O$.\end{cor}
\begin{proof}
Let $M$ be the centralizer of $K$ in $\mathbf{Sp}$, then $(K,M)$
is also a reductive dual pair \cite[Fact 1]{Howe1989transcending}.
 For the dual pair $(K,M)$ and any $\widetilde{K}$-type $\sigma$,
Proposition \ref{prop:either0} asserts that $\mathrm{Hom}_{\widetilde{K}}(\mathcal{F}_{i},\sigma)=0$
for some $i\in\{0,1\}$. Then $m(\sigma,\mathcal{F}_{i})=0$ and $\sigma\notin\mathcal{R}(\widetilde{K},\mathcal{F}_{i})$.
\end{proof}
For $\sigma\in\mathcal{R}(\widetilde{K},\mathcal{F})$, \cite{Howe1989transcending}
defines the \emph{degree} $\mathrm{deg}(\sigma)$ with respect to
$(G,G')$ as the minimal degree of polynomials in the \emph{$\sigma$-isotypic
subspace} $\mathcal{F}_{\sigma}=\sum_{\varphi\in\mathrm{Hom}_{\widetilde{K}}(\sigma,\mathcal{F})}\mathrm{Im}(\varphi)$.

\begin{cor}
\label{cor:deg_i}If $\sigma\in\mathcal{R}(\widetilde{K},\mathcal{F}_{i})$
for some $i\in\{0,1\}$, then $\mathrm{deg}(\sigma)\equiv i$ $(\mathrm{mod}\ 2)$.\end{cor}
\begin{proof}
By Corollary \ref{cor:disjoint}, $\sigma\notin\mathcal{R}(\widetilde{K},\mathcal{F}_{1-i})$.
So $\mathcal{F}_{\sigma}\subseteq\mathcal{F}_{i}$ and $\mathrm{deg}(\sigma)\equiv i$
$(\mathrm{mod}\ 2)$.\end{proof}
\begin{thm}
\label{thm:deg-parity} For a real reductive dual pair $(G,G')$,
let $\pi$ be an irreducible $(\mathfrak{g},\widetilde{K})$-module
with a non-zero theta lift. If $\sigma_{1}$, $\sigma_{2}\in\mathcal{R}(\widetilde{K},\pi)$,
then $\sigma_{1}$, $\sigma_{2}\in\mathcal{R}(\widetilde{K},\mathcal{F})$
and $\mathrm{deg}(\sigma_{1})\equiv\mathrm{deg}(\sigma_{2})$ $(\mathrm{mod}\ 2)$.\end{thm}
\begin{proof}
 Let $V_{\pi}$ be an underlying space of $\pi$. By definition
$\mathrm{Hom}_{\mathfrak{g},\widetilde{K}}(\mathcal{F}_{i},V_{\pi})\neq0$
for some $i\in\{0,1\}$. Take a non-zero $T\in\mathrm{Hom}_{\mathfrak{g},\widetilde{K}}(\mathcal{F}_{i},V_{\pi})$.
As $\pi$ is an irreducible $(\mathfrak{g},\widetilde{K})$-module,
the image $T(\mathcal{F}_{i})=V_{\pi}$.

For $j\in\{0,1\}$, let $W_{\sigma_{j}}$ be an underlying space of
$\sigma_{j}$. As $\dim\mathrm{Hom}_{\widetilde{K}}(V_{\pi},W_{\sigma_{j}})=m(\sigma_{j},V_{\pi})>0$,
take a non-zero $\varphi_{j}\in\mathrm{Hom}_{\widetilde{K}}(V_{\pi},W_{\sigma_{j}})$.
As $W_{\sigma_{j}}$ is an irreducible $\widetilde{K}$-module, the
image $\varphi_{j}(V_{\pi})=W_{\sigma_{j}}$. Now the surjective $\varphi_{j}\circ T:\mathcal{F}_{i}\to W_{\sigma_{j}}$
gives a non-zero element of $\mathrm{Hom}_{\widetilde{K}}(\mathcal{F}_{i},W_{\sigma_{j}})$.
So $m(\sigma_{j},\mathcal{F}_{i})>0$, $\sigma_{j}\in\mathcal{R}(\widetilde{K},\mathcal{F}_{i})\subseteq\mathcal{R}(\widetilde{K},\mathcal{F})$,
and $\mathrm{deg}(\sigma_{j})\equiv i$ $(\mathrm{mod}\ 2)$ for both
$j\in\{0,1\}$ by Corollary \ref{cor:deg_i}.\end{proof}
\begin{rem*}
Corollary \ref{cor:disjoint} (and the degree-parity preservation
of Theorem \ref{thm:deg-parity} as a consequence) can also be deduced
from classical invariant theory, similar to the proof of \cite[Lemma 6]{Fan2017explicit}
based on \cite[(3.9)(b)]{Howe1989transcending} and \cite{Howe1989remark}.
\end{rem*}

\section{Parity Preservation}

From Theorem \ref{thm:deg-parity}, we deduce the preservation of
parity of all $K$-types occurring in an arbitrary irreducible $(\mathfrak{g},K)$-module
of a Lie group $G$ in real reductive dual pairs.
\begin{thm}
[Parity preservation]\label{thm:Main} Let $G$ and $K$ be as in
the following table, with  $K$ embedded in $G$ as a maximal compact
subgroup in the usual way. If $\pi$ is an irreducible $(\mathfrak{g},K)$-module
and $\sigma_{1}$, $\sigma_{2}\in\mathcal{R}(K,\pi)$, then $\varepsilon(\sigma_{1})=\varepsilon(\sigma_{2})$,
where $\varepsilon:\mathcal{R}(K)\to\mathbb{Z}/2\mathbb{Z}$ is the
parity of $K$-types defined explicitly in the next subsection.
\end{thm}
\begin{center}

\begin{tabular}{c|c|c|c|c|c|c|c}
\hline 
$G$ & $K$ & $\quad$ & $G$ & $K$ & $\quad$ & $G$ & $K$\tabularnewline
\cline{1-2} \cline{4-5} \cline{7-8} 
$GL_{m}(\mathbb{R})$ & $O(m)$ &  & $O_{p}(\mathbb{C})$ & $O(p)$  &  & $Sp_{2n}(\mathbb{C})$ & $Sp(n)$\tabularnewline
$GL_{m}(\mathbb{C})$ & $U(m)$ &  & $O(p,q)$ & $O(p)\times O(q)$ &  & $Sp_{2n}(\mathbb{R})$ & $U(n)$\tabularnewline
$GL_{m}(\mathbb{H})$ & $Sp(m)$ &  & $Sp(p,q)$ & $Sp(p)\times Sp(q)$ &  & $O^{*}(2n)$ & $U(n)$\tabularnewline
 &  &  & $U(p,q)$ & $U(p)\times U(q)$ &  &  & \tabularnewline
\hline 
\end{tabular}

\end{center}

\subsection{Parametrization and parity for $K$-types}

For $K=U(n)$, $T=U(1)^{n}=\mathrm{diag}(U(1),\dots,U(1))$ is a maximal
torus, with the standard system of positive roots $\Delta^{+}(\mathfrak{t},\mathfrak{k})=\{e_{i}-e_{j}\mid1\leqslant i<j\leqslant n\}$.
Write $\mathfrak{t}_{0}$ for the real Lie algebra of $T$. Write
each weight in $\sqrt{-1}\mathfrak{t}_{0}^{*}$ as the $n$-tuple
of coefficients under the basis $e_{1}$, $\dots$, $e_{n}$. Then
a $U(n)$-type is parametrized by its highest weight $(a_{1},\dots,a_{n})$
with integers $a_{1}\geqslant a_{2}\geqslant\cdots\geqslant a_{n}$.

For $K=Sp(n)$, $T=Sp(1)^{n}=\mathrm{diag}(Sp(1),\dots,Sp(1))$ is
a maximal torus, with the standard system of positive roots $\Delta^{+}(\mathfrak{t},\mathfrak{k})=\{e_{i}\pm e_{j}\mid1\leqslant i<j\leqslant n\}\cup\{2e_{i}\mid1\leqslant i\leqslant n\}$.
Write each weight in $\sqrt{-1}\mathfrak{t}_{0}^{*}$ as the $n$-tuple
of coefficients under the basis $e_{1}$, $\dots$, $e_{n}$. Then
a $Sp(n)$-type is parametrized by its highest weight $(a_{1},\dots,a_{n})$
with integers $a_{1}\geqslant a_{2}\geqslant\cdots\geqslant a_{n}\geqslant0$.

\begin{lem}
[\cite{Weyl1997classical}] \label{lem:sign} Embed $O(n)$ in $U(n)$
as $O(n)=U(n)\cap GL(n,\mathbb{R})$. There is a bijection $\mathcal{R}(O(n))\stackrel{\simeq}{\longleftrightarrow}\{(b_{1},b_{2},\dots,b_{r},\underbrace{1,\dots,1}_{s},\underbrace{0,\dots,0}_{n-r-s})\in\mathcal{R}(U(n)):2r+s\leqslant n,\ b_{r}\geqslant2\}$,
such that an $O(n)$-type $\sigma$ corresponds to a $U(n)$-type
$\lambda$ if and only if the highest weight vectors of $\lambda$
generate an $O(n)$-module equivalent to $\sigma$. Therefore, an
$O(n)$-type $\sigma$ can be parametrized as follows (with $[t]$
the greatest integer less than or equal to $t$):
\[
\begin{cases}
(b_{1},b_{2},\dots,b_{r},\underbrace{1,\dots,1}_{s},\underbrace{0,\dots,0}_{[\frac{n}{2}]-r-s};+1) & \text{if }r+s\leqslant\frac{n}{2},\\
(b_{1},b_{2},\dots,b_{r},\underbrace{1,\dots,1}_{n-2r-s},\underbrace{0,\dots,0}_{[\frac{n}{2}]-n+r+s};-1) & \text{if }\frac{n}{2}\leqslant r+s\leqslant n-r.
\end{cases}
\]
\end{lem}
\begin{rem*}
When $r+s=\frac{n}{2}$, these two cases coincide and give the same
$\sigma$.
\end{rem*}
An $O(n)$-type is parametrized as $(\underbrace{a_{1},a_{2},\dots,a_{x},0,\dots,0}_{[\frac{n}{2}]};\epsilon)$
with integers $a_{1}\geqslant a_{2}\geqslant\cdots\geqslant a_{x}\geqslant1$,
$\epsilon\in\{\pm1\}$, and corresponding $U(p)$-type $(a_{1},a_{2},\dots,a_{x},\underbrace{1,\dots,1}_{\frac{1-\epsilon}{2}(n-2x)},0,\dots,0).$
When $n$ is even and $n=2x$, the two choices of $\epsilon\in\{\pm1\}$
give the same $O(n)$-type.

Define the \emph{parity} $\varepsilon:\mathcal{R}(K)\to\mathbb{Z}/2\mathbb{Z}$
for $K=U(n)$, $Sp(n)$ or $O(n)$ as:

\medskip{}\noindent\resizebox{\textwidth}{!}{%

\begin{tabular}{c|c|c}
\hline 
$K$ & a $K$-type $\sigma$ is parametrized as & parity $\varepsilon(\sigma)\in\mathbb{Z}/2\mathbb{Z}$\tabularnewline
\hline 
\hline 
\multirow{1}{*}{$U(n)$} & $(a_{1},\dots,a_{n})$ with integers $a_{1}\geqslant a_{2}\geqslant\cdots\geqslant a_{n}$  & \multirow{2}{*}{$\sum_{i=1}^{n}a_{i}$ $(\mathrm{mod\ }2)$}\tabularnewline
\cline{1-2} 
\multirow{1}{*}{$Sp(n)$} & $(a_{1},\dots,a_{n})$ with integers $a_{1}\geqslant a_{2}\geqslant\cdots\geqslant a_{n}\geqslant0$ & \tabularnewline
\hline 
\multirow{2}{*}{$O(n)$} & $(a_{1},\dots,a_{[\frac{n}{2}]};\epsilon)$ with integers  & $\sum_{i=1}^{[\frac{n}{2}]}a_{i}+\frac{1-\epsilon}{2}\cdot n$\tabularnewline
 & $a_{1}\geqslant a_{2}\geqslant\cdots\geqslant a_{[\frac{n}{2}]}\geqslant0$,
and $\epsilon\in\{\pm1\}$ & $(\mathrm{mod\ }2)$\tabularnewline
\hline 
\end{tabular}

}\medskip{}
\begin{rem*}
The parity of an $O(n)$-type is the same as that of its corresponding
$U(p)$-type.
\end{rem*}
For $K=K_{1}\times\cdots\times K_{r}$ with $K_{i}=U(n_{i})$, $Sp(n_{i})$
or $O(n_{i})$, as $\mathcal{R}(K)=\{\bigotimes_{i=1}^{r}\sigma_{i}:\sigma_{i}\in\mathcal{R}(K_{i})\}$,
define the parity $\varepsilon:\mathcal{R}(K)\to\mathbb{Z}/2\mathbb{Z}$
by $\varepsilon(\bigotimes_{i=1}^{r}\sigma_{i})=\sum_{i=1}^{r}\varepsilon(\sigma_{i})\in\mathbb{Z}/2\mathbb{Z}$.

\subsection{Non-vanishing and splitting conditions}

To prove Theorem \ref{thm:Main}, we recall the non-vanishing and
splitting conditions for the local theta correspondence over $\mathbb{R}$.

Let $W$ be a real symplectic vector space. A reductive dual pair
$(G,G')$ in $Sp(W)$ is called \emph{irreducible} if $G\cdot G'$
acts irreducibly on $W$. Each reductive dual pair $(G,G')$ in $Sp(W)$
can be decomposed into a direct sum of irreducible pairs, namely,
there is an orthogonal direct sum decomposition $W=\bigoplus_{i=1}^{k}W_{i}$
such that $G\cdot G'$ acts irreducibly on $W_{i}$, and the restrictions
of actions of $(G,G')$ to $W_{i}$ define irreducible reductive dual
pairs $(G_{i},G_{i}')$ in $Sp(W_{i})$. Indeed, $G=G_{1}\times\cdots\times G_{k}$
and $G'=G_{1}'\times\cdots\times G_{k}'$. All irreducible real reductive
dual pairs are classified in the following table (c.f. \cite{Howe1979theta}).

\medskip{}
\noindent\resizebox{\textwidth}{!}{%

\begin{tabular}{c|c|c|c|c}
\hline 
Type I & $\mathbf{Sp}$ & \multirow{5}{*}{} & Type II & $\mathbf{Sp}$\tabularnewline
\cline{1-2} \cline{4-5} 
$(O(p,q),Sp(2n,\mathbb{R}))$ & $Sp(2n(p+q),\mathbb{R})$ &  & $(GL_{m}(\mathbb{R}),GL_{n}(\mathbb{R}))$ & $Sp(2mn,\mathbb{R})$\tabularnewline
$(O_{p}(\mathbb{C}),Sp(2n,\mathbb{C}))$ & $Sp(4pn,\mathbb{R})$ &  & \multirow{1}{*}{$(GL_{m}(\mathbb{C}),GL_{n}(\mathbb{C}))$} & \multirow{1}{*}{$Sp(4mn,\mathbb{R})$}\tabularnewline
$(Sp(p,q),O^{*}(2n))$ & $Sp(4n(p+q),\mathbb{R})$ &  & $(GL_{m}(\mathbb{H}),GL_{n}(\mathbb{H}))$ & $Sp(8mn,\mathbb{R})$\tabularnewline
$(U(p,q),U(r,s))$ & $Sp(2(p+q)(r+s),\mathbb{R})$ &  &  & \tabularnewline
\hline 
\end{tabular}

}\medskip{}

An irreducible real reductive dual pair $(G,G')$ of type I is said
to be \emph{in the stable range} with $G$ the smaller member if the
defining module of $G'$ contains an isotropic subspace of the same
dimension as that of the defining module of $G$. All irreducible
real reductive dual pairs $(G_{1},G_{2})$ in the stable range are
listed in following table.

\medskip{}\begin{center}

\begin{tabular}{c|c|c|c}
\hline 
\multirow{1}{*}{$G_{1}$} & \multirow{1}{*}{$G_{2}$} & with $G_{1}$ smaller & with $G_{2}$ smaller\tabularnewline
\hline 
\hline 
$O_{p}(\mathbb{C})$ & $Sp_{2n}(\mathbb{C})$ & $n\geqslant p$ & $p\geqslant4n$\tabularnewline
\hline 
$O(p,q)$ & $Sp_{2n}(\mathbb{R})$ & $n\geqslant p+q$ & $p,q\geqslant2n$\tabularnewline
\hline 
$Sp(p,q)$ & $O^{*}(2n)$ & $n\geqslant2(p+q)$ & $p,q\geqslant n$\tabularnewline
\hline 
$U(p,q)$ & $U(r,s)$ & $r,s\geqslant p+q$ & $p,q\geqslant r+s$\tabularnewline
\hline 
\end{tabular}

\end{center}\medskip{}

Write the two elements of $\mathrm{Ker}(\widetilde{\mathbf{Sp}}\to\mathbf{Sp})=\mu_{2}$
as $e$ and $-e$, such that $e=(-e)^{2}$ is the identity element.
A $(\mathfrak{g},\widetilde{K})$-module is called \emph{genuine}
if $-e$ acts on it as the scalar multiplication by $-1$. Clearly,
an irreducible $(\mathfrak{g},\widetilde{K})$-module $\pi$ with
a non-zero theta lift must be genuine, since $\omega(-e)$ acts on
$\mathcal{F}$ by the scalar $-1$. Conversely, two lemmas hold:
\begin{lem}
[{Non-vanishing of theta liftings in the stable range \cite{ProtsakPrzebinda2008occurrence}}]\label{lem:non-van-I}
If $(G,G')$ is an irreducible real reductive dual pair of type I
in the stable range with $G$ the smaller member, then each genuine
irreducible $(\mathfrak{g},\widetilde{K})$-module has a non-zero
theta lift.
\end{lem}

\begin{lem}
[{\cite{Moeglin1989correspondance,AdamsBarbasch1995reductive,LiPaulTanZhu2003explicit}}]\label{lem:non-van-II}
For $(G,G')=(GL_{m}(F),GL_{n}(F))$ with $F=\mathbb{R}$, $\mathbb{C}$
or $\mathbb{H}$, if $n\geqslant m$, then each genuine irreducible
$(\mathfrak{g},\widetilde{K})$-module has a non-zero theta lift.
\end{lem}
Suppose that the covering $\widetilde{G}\to G$ \emph{splits},  namely,
there exists an embedding $G\hookrightarrow\widetilde{G}$ such that
the composition $G\hookrightarrow\widetilde{G}\to G$ is the identity
map on $G$. We may identify $G$ as a subgroup of $\widetilde{G}$
via this embedding. Then $\widetilde{G}=G\times\mu_{2}$ in the sense
that the two subgroups $G$ and $\mu_{2}=\mathrm{Ker}(\widetilde{G}\to G)$
commute, generate $\widetilde{G}$, and $G\cap\mu_{2}=\{e\}$. Similarly
we have $\widetilde{K}=K\times\mu_{2}$.     An irreducible
$(\mathfrak{g},K)$-module $\pi$ gives rise to a genuine irreducible
$(\mathfrak{g},\widetilde{K})$-module $\tilde{\pi}$, with the same
underlying space and actions of $(\mathfrak{g},K)$, while $\tilde{\pi}(-e)$
acts by the scalar $-1$. Similarly, a $K$-type $\sigma$ gives rise
to a genuine $\tilde{K}$-type $\tilde{\sigma}$, with the same underlying
space and actions of $K$, while $\tilde{\sigma}(-e)$ acts by the
scalar $-1$. Consider the actions of $K$ on the Fock model $\mathcal{F}$
via the embedding $K\hookrightarrow\widetilde{K}=K\times\mu_{2}$.
Clearly,  
\[
\sigma\in\mathcal{R}(K,\mathcal{F})\iff\tilde{\sigma}\in\mathcal{R}(\widetilde{K},\mathcal{F}).
\]
In that case we define the \emph{degree} $\mathrm{deg}(\sigma)$ for
$\sigma\in\mathcal{R}(K,\mathcal{F})$ by  $\deg(\sigma)=\deg(\tilde{\sigma})$.
\begin{prop}
\label{Prop:non-v-deg-parity}Let $(G,G')$ be an irreducible real
reductive dual pair, either in the stable range or of type II, with
$G$ the smaller member. Suppose that \textup{$\widetilde{G}\to G$}
splits over $G$. If $\pi$ is an irreducible $(\mathfrak{g},K)$-module,
and $\sigma_{1}$, $\sigma_{2}\in\mathcal{R}(K,\pi)$, then $\sigma_{1}$,
$\sigma_{2}\in\mathcal{R}(K,\mathcal{F})$ and $\mathrm{deg}(\sigma_{1})\equiv\mathrm{deg}(\sigma_{2})\ (\mathrm{mod}\ 2)$.\end{prop}
\begin{proof}
As we said, $\pi$ gives rise to a genuine irreducible $(\mathfrak{g},\widetilde{K})$-module
$\tilde{\pi}$, while $\sigma_{1}$ and $\sigma_{2}\in\mathcal{R}(K,\pi)$
give rise to $\tilde{\sigma}_{1}$ and $\tilde{\sigma}_{2}\in\mathcal{R}(\widetilde{K},\widetilde{\pi})$.
By Lemma \ref{lem:non-van-I}, \ref{lem:non-van-II}, $\theta(\tilde{\pi})\neq0$.
By Theorem \ref{thm:deg-parity}, $\tilde{\sigma}_{1}$, $\tilde{\sigma}_{2}\in\mathcal{R}(\widetilde{K},\mathcal{F})$,
and $\mathrm{deg}(\tilde{\sigma}_{1})\equiv\mathrm{deg}(\tilde{\sigma}_{2})\ (\mathrm{mod}\ 2)$.
Equivalently, we have $\sigma_{1}$, $\sigma_{2}\in\mathcal{R}(K,\mathcal{F})$
and $\mathrm{deg}(\sigma_{1})\equiv\mathrm{deg}(\sigma_{2})\ (\mathrm{mod}\ 2)$.
\end{proof}
 For an irreducible real reductive dual pair $(G_{1},G_{2})$, The
following table gives some sufficient conditions for $\widetilde{G}_{i}\to G_{i}$
to split (c.f. \cite{AdamsBarbasch1995reductive,AdamsBarbasch1998genuine,Paul1998howe,Adams2007theta}).

\medskip{}\begin{center}

\begin{tabular}{c|c|c|c}
\hline 
\multirow{1}{*}{$G_{1}$} & \multirow{1}{*}{$G_{2}$} & $\widetilde{G}_{1}\to G_{1}$ splits & $\widetilde{G}_{2}\to G_{2}$ splits\tabularnewline
\hline 
\hline 
$O_{p}(\mathbb{C})$ & $Sp_{2n}(\mathbb{C})$ & \multicolumn{2}{c}{always}\tabularnewline
\hline 
$O(p,q)$ & $Sp_{2n}(\mathbb{R})$ & if $n$ is even & if $p+q$ is even\tabularnewline
\hline 
$Sp(p,q)$ & $O^{*}(2n)$ & \multicolumn{2}{c}{always}\tabularnewline
\hline 
$U(p,q)$ & $U(r,s)$ & if $r+s$ is even & if $p+q$ is even\tabularnewline
\hline 
$GL_{m}(\mathbb{R})$ & $GL_{n}(\mathbb{R})$ & if $n$ is even & if $m$ is even\tabularnewline
\hline 
$GL_{m}(\mathbb{C})$ & $GL_{n}(\mathbb{C})$ & \multicolumn{2}{c}{always}\tabularnewline
\hline 
$GL_{m}(\mathbb{H})$ & $GL_{n}(\mathbb{H})$ & \multicolumn{2}{c}{always}\tabularnewline
\hline 
\end{tabular}

\end{center}\medskip{}

Let $(G,G')$ be an irreducible real reductive dual pair with $\widetilde{G}\to G$
splitting. The following table lists $\deg(\sigma)$ explicitly for
$\sigma\in\mathcal{R}(K,\mathcal{F})$ (c.f. \cite{Moeglin1989correspondance,AdamsBarbasch1995reductive,Paul1998howe,LiPaulTanZhu2003explicit,Paul2005howe}).

\medskip{}\noindent\resizebox{\textwidth}{!}{%

\begin{tabular}{c|c|c|c|l}
\hline 
$G$ & $G'$ & $K$ & $\sigma\in\mathcal{R}(K,\mathcal{F})$ & $\qquad\qquad\mathrm{deg}(\sigma)$\tabularnewline
\hline 
\hline 
$O_{p}(\mathbb{C})$ & $Sp_{2n}(\mathbb{C})$ & $O(p)$ & $(a_{1},\dots,a_{[\frac{p}{2}]};\epsilon)$ & $\sum_{i=1}^{[\frac{p}{2}]}a_{i}+\frac{1-\epsilon}{2}(p-2|\{i:a_{i}>0\}|)$\tabularnewline
\hline 
$Sp_{2n}(\mathbb{C})$ & $O_{p}(\mathbb{C})$ & $Sp(n)$ & $(a_{1},\dots,a_{n})$ & $\sum_{i=1}^{n}a_{i}$\tabularnewline
\hline 
\multirow{2}{*}{$O(p,q)$} & \multirow{2}{*}{$Sp_{2n}(\mathbb{R})$} & $O(p)$ & $(a_{1},\dots,a_{[\frac{p}{2}]};\epsilon)$ & $\sum_{i=1}^{[\frac{p}{2}]}a_{i}+\frac{1-\epsilon}{2}(p-2|\{i:a_{i}>0\}|)$\tabularnewline
 &  & $\times O(q)$ & $\otimes(b_{1},\dots,b_{[\frac{q}{2}]};\eta)$ & $+\sum_{j=1}^{[\frac{q}{2}]}b_{j}+\frac{1-\eta}{2}(q-2|\{j:b_{j}>0\}|)$\tabularnewline
\hline 
$Sp_{2n}(\mathbb{R})$ & $O(p,q)$ & $U(n)$ & $(a_{1},\dots,a_{n})$ & $\sum_{i=1}^{n}|a_{i}-\frac{p-q}{2}|$\tabularnewline
\hline 
\multirow{2}{*}{$Sp(p,q)$} & \multirow{2}{*}{$O^{*}(2n)$} & $Sp(p)$ & $(a_{1},\dots,a_{p})$ & \multirow{2}{*}{$\sum_{i=1}^{p}a_{i}+\sum_{j=1}^{q}b_{j}$}\tabularnewline
 &  & $\times Sp(q)$ & $\otimes(b_{1},\dots,b_{q})$ & \tabularnewline
\hline 
$O^{*}(2n)$ & $Sp(p,q)$ & $U(n)$ & $(a_{1},\dots,a_{n})$ & $\sum_{i=1}^{n}|a_{i}-p+q|$\tabularnewline
\hline 
\multirow{2}{*}{$U(p,q)$} & \multirow{2}{*}{$U(r,s)$} & $U(p)$ & $(a_{1},\dots,a_{p})$ & \multirow{2}{*}{$\sum_{i=1}^{p}|a_{i}-\frac{r-s}{2}|$$+\sum_{j=1}^{q}|b_{j}+\frac{r-s}{2}|$}\tabularnewline
 &  & $\times U(q)$ & $\otimes(b_{1},\dots,b_{q})$ & \tabularnewline
\hline 
$GL_{m}(\mathbb{R})$ & $GL_{n}(\mathbb{R})$ & $O(m)$ & $(a_{1},\dots,a_{[\frac{m}{2}]};\epsilon)$ & $\sum_{i=1}^{[\frac{m}{2}]}a_{i}+\frac{1-\epsilon}{2}(m-2|\{i:a_{i}>0\}|)$\tabularnewline
\hline 
$GL_{m}(\mathbb{C})$ & $GL_{n}(\mathbb{C})$ & $U(m)$ & $(a_{1},\dots,a_{m})$ & $\sum_{i=1}^{m}|a_{i}|$\tabularnewline
\hline 
$GL_{m}(\mathbb{H})$ & $GL_{n}(\mathbb{H})$ & $Sp(m)$ & $(a_{1},\dots,a_{m})$ & $\sum_{i=1}^{m}a_{i}$\tabularnewline
\hline 
\end{tabular}

}\medskip{}

\subsection{Proof of Theorem \ref{thm:Main}}

By Proposition \ref{Prop:non-v-deg-parity}, it suffices to find
a suitable $G'$ such that $(G,G')$ is an irreducible real reductive
dual pair satisfying three conditions:
\begin{description}
\item [{(1)}] It is either in the stable range or of type II, with $G$
the smaller member.
\item [{(2)}] The covering $\widetilde{G}\to G$ splits.
\item [{(3)}] For $(G,G')$, the degree $\deg(\sigma)\equiv\varepsilon(\sigma)\ (\mathrm{mod}\ 2)$
for any $\sigma\in\mathcal{R}(K,\mathcal{F})$.
\end{description}
We can take $G'$ according to the following table, which lists explicit
sufficient conditions to ensure \textbf{(1)}, \textbf{(2)} and \textbf{(3)}.

\medskip{}\begin{center}

\begin{tabular}{c|c|c|c|c}
\hline 
$G$ & $G'$ & condition (1) & condition (2) & condition (3)\tabularnewline
\hline 
\hline 
$O_{p}(\mathbb{C})$ & $Sp_{2n}(\mathbb{C})$ & $n\geqslant p$ &  & \tabularnewline
\hline 
$Sp_{2n}(\mathbb{C})$ & $O_{p}(\mathbb{C})$ & $p\geqslant4n$ &  & \tabularnewline
\hline 
$O(p,q)$ & $Sp_{2n}(\mathbb{R})$ & $n\geqslant p+q$ & $2\mid n$ & \tabularnewline
\hline 
$Sp_{2n}(\mathbb{R})$ & $O(p,q)$ & $p,q\geqslant2n$ & $2\mid p+q$ & $4\mid p-q$\tabularnewline
\hline 
$Sp(p,q)$ & $O^{*}(2n)$ & $n\geqslant2(p+q)$ &  & \tabularnewline
\hline 
$O^{*}(2n)$ & $Sp(p,q)$ & $p,q\geqslant n$ &  & $2\mid p-q$\tabularnewline
\hline 
$U(p,q)$ & $U(r,s)$ & $r,s\geqslant p+q$ & $2\mid r+s$ & $4\mid r-s$\tabularnewline
\hline 
$GL_{m}(\mathbb{R})$ & $GL_{n}(\mathbb{R})$ & $n\geqslant m$ & $2\mid n$ & \tabularnewline
\hline 
$GL_{m}(\mathbb{C})$ & $GL_{n}(\mathbb{C})$ & $n\geqslant m$ &  & \tabularnewline
\hline 
$GL_{m}(\mathbb{H})$ & $GL_{n}(\mathbb{H})$ & \multirow{1}{*}{$n\geqslant m$} &  & \tabularnewline
\hline 
\end{tabular}

\end{center}\medskip{}

\subsection{Generalization}

Theorem \ref{thm:Main} can be generalized to all members of (reducible)
real reductive dual pairs.
\begin{thm}
\label{thm:gen} Let $G$ be a member of a real reductive dual pair,
with a maximal compact subgroup $K$. If $\pi$ is an irreducible
$(\mathfrak{g},K)$-module, and $\sigma$, $\sigma'\in\mathcal{R}(K,\pi)$,
then $\varepsilon(\sigma)=\varepsilon(\sigma')$.\end{thm}
\begin{proof}
Any real reductive dual pair is a direct sum of irreducible ones,
so $G=G_{1}\times G_{2}\times\cdots\times G_{r}$ with each $G_{i}$
a member of an irreducible real reductive dual pair. Then $K=K_{1}\times\cdots\times K_{r}$
with $K_{i}$ a maximal compact subgroup of $G_{i}$. By \cite{GourevitchKemarsky2012irreducible},
$\pi=\bigotimes_{i=1}^{r}\pi_{i}$, where $\pi_{i}$ is an irreducible
$(\mathfrak{g}_{i},K_{i})$-module. Moreover, $\sigma=\bigotimes_{i=1}^{r}\sigma_{i}$
and $\sigma'=\bigotimes_{i=1}^{r}\sigma_{i}'$, with $\sigma_{i}$
and $\sigma_{i}'\in\mathcal{R}(K_{i},\pi_{i})$. Theorem \ref{thm:Main}
holds for $(G_{i},K_{i})$ (up to isomorphisms). So $\varepsilon(\sigma_{i})=\varepsilon(\sigma'_{i})$
for all $i$. Therefore, $\varepsilon(\sigma)=\sum_{i=1}^{r}\varepsilon(\sigma_{i})=\sum_{i=1}^{r}\varepsilon(\sigma_{i}')=\varepsilon(\sigma')$.
\end{proof}
In the end, please note that the phenomenon of parity preservation
of $K$-types in an irreducible admissible representation is well-known
to experts for many kinds of Lie groups for many years. Many cases
can be proved in a more elementary way, without the help of Howe's
theory of theta correspondence. If $G$ is connected with $\mathrm{rank}(G)=\mathrm{rank}(K)$
and of type A, C or D (for example $G=U(p,q)$ or $Sp(2n,\mathbb{R})$),
this follows easily from the fact that the difference between any
two (highest) weights in an irreducible $(\mathfrak{g},K)$-module
is a sum of roots, which are ``even'' in our sense of parity. However,
in other cases, especially when $G$ is disconnected (so that the
definition of ``parity'' is less natural), our approach makes the
phenomenon much clearer.

This note extends parts of the author\textquoteright s doctoral thesis
work in HKUST. It was written partially while the author was visiting
the Institute for Mathematical Sciences, National University of Singapore
in 2016. The visit was supported by the Institute. It was completed
in Sun Yat-sen University. The author appreciates the referees for
suggesting many improvements of the manuscript.

\bibliographystyle{amsalpha}
\bibliography{bibfile}

\providecommand{\bysame}{\leavevmode\hbox to3em{\hrulefill}\thinspace}
\providecommand{\MR}{\relax\ifhmode\unskip\space\fi MR }
\providecommand{\MRhref}[2]{%
  \href{http://www.ams.org/mathscinet-getitem?mr=#1}{#2}
}
\providecommand{\href}[2]{#2}
\begin{thebibliography}{LPTZ03}

\bibitem[AB95]{AdamsBarbasch1995reductive}
Jeffrey Adams and Dan Barbasch, \emph{Reductive dual pair correspondence for
  complex groups}, J. Funct. Anal. \textbf{132} (1995), no.~1, 1--42.
  \MR{1346217}

\bibitem[AB98]{AdamsBarbasch1998genuine}
\bysame, \emph{Genuine representations of the metaplectic group}, Compositio
  Math. \textbf{113} (1998), no.~1, 23--66. \MR{1638210}

\bibitem[Ada07]{Adams2007theta}
Jeffrey Adams, \emph{The theta correspondence over {$\mathbb{R}$}}, Harmonic
  analysis, group representations, automorphic forms and invariant theory,
  Lect. Notes Ser. Inst. Math. Sci. Natl. Univ. Singap., vol.~12, World Sci.
  Publ., Hackensack, NJ, 2007, pp.~1--39. \MR{2401808}

\bibitem[Fan17]{Fan2017explicit}
Xiang Fan, \emph{Explicit induction principle and symplectic-orthogonal theta
  lifts}, J. Funct. Anal. \textbf{273} (2017), no.~11, 3504--3548. \MR{3706610}

\bibitem[GK13]{GourevitchKemarsky2012irreducible}
Dmitry Gourevitch and Alexander Kemarsky, \emph{Irreducible representations of
  a product of real reductive groups}, J. Lie Theory \textbf{23} (2013), no.~4,
  1005--1010. \MR{3185208}

\bibitem[How79]{Howe1979theta}
Roger Howe, \emph{{$\theta $}-series and invariant theory}, Automorphic forms,
  representations and {$L$}-functions ({P}roc. {S}ympos. {P}ure {M}ath.,
  {O}regon {S}tate {U}niv., {C}orvallis, {O}re., 1977), {P}art 1, Proc. Sympos.
  Pure Math., XXXIII, Amer. Math. Soc., Providence, R.I., 1979, pp.~275--285.
  \MR{546602}

\bibitem[How89a]{Howe1989remark}
\bysame, \emph{Remarks on classical invariant theory}, Trans. Amer. Math. Soc.
  \textbf{313} (1989), no.~2, 539--570. \MR{986027}

\bibitem[How89b]{Howe1989transcending}
\bysame, \emph{Transcending classical invariant theory}, J. Amer. Math. Soc.
  \textbf{2} (1989), no.~3, 535--552. \MR{985172}

\bibitem[LPTZ03]{LiPaulTanZhu2003explicit}
Jian-Shu Li, Annegret Paul, Eng-Chye Tan, and Chen-Bo Zhu, \emph{The explicit
  duality correspondence of {$(Sp(p,q), O^\ast(2n))$}}, J. Funct. Anal.
  \textbf{200} (2003), no.~1, 71--100. \MR{1974089}

\bibitem[LV80]{LionVergne1980theweil}
G{\'e}rard Lion and Mich{\`e}le Vergne, \emph{{The Weil representation, Maslov
  index and Theta series}}, Progress in Mathematics, vol.~6, Birkh\"auser,
  Boston, Mass., 1980. \MR{573448}

\bibitem[M{\oe}g89]{Moeglin1989correspondance}
Colette M{\oe}glin, \emph{Correspondance de {H}owe pour les paires reductives
  duales: quelques calculs dans le cas archim\'edien}, J. Funct. Anal.
  \textbf{85} (1989), no.~1, 1--85. \MR{1005856}

\bibitem[Pau98]{Paul1998howe}
Annegret Paul, \emph{Howe correspondence for real unitary groups}, J. Funct.
  Anal. \textbf{159} (1998), no.~2, 384--431. \MR{1658091}

\bibitem[Pau05]{Paul2005howe}
\bysame, \emph{On the {H}owe correspondence for symplectic-orthogonal dual
  pairs}, J. Funct. Anal. \textbf{228} (2005), no.~2, 270--310. \MR{2175409}

\bibitem[PP08]{ProtsakPrzebinda2008occurrence}
Victor Protsak and Tomasz Przebinda, \emph{On the occurrence of admissible
  representations in the real {H}owe correspondence in stable range},
  Manuscripta Math. \textbf{126} (2008), no.~2, 135--141. \MR{2403182}

\bibitem[Sha62]{Shale1962linear}
David Shale, \emph{Linear symmetries of free boson fields}, Trans. Amer. Math.
  Soc. \textbf{103} (1962), 149--167. \MR{0137504}

\bibitem[Wei64]{Weil1964certains}
Andr\'e Weil, \emph{Sur certains groupes d'op\'erateurs unitaires}, Acta Math.
  \textbf{111} (1964), 143--211. \MR{0165033}

\bibitem[Wey97]{Weyl1997classical}
Hermann Weyl, \emph{The classical groups: their invariants and
  representations}, Princeton Landmarks in Mathematics, Princeton University
  Press, Princeton, NJ, 1997, Fifteenth printing, Princeton Paperbacks.
  \MR{1488158}

\end{thebibliography}

\end{document}